\documentclass[12pt]{amsart}

\usepackage{amsmath,amssymb,amsthm,latexsym}
\usepackage[dvips]{graphicx}
\usepackage[T1]{fontenc}
\usepackage[centertags]{amsmath}
\usepackage{amsfonts}
\usepackage{amscd}
\usepackage{euscript}
\usepackage{graphicx}
\usepackage{color}
\usepackage[all]{xy}
\usepackage[T1]{fontenc}
\usepackage{parskip}
\usepackage{enumitem}

\newcommand{\B}[1]{{\mathbf #1}}

\makeatletter
\def\@settitle{\begin{center}%
  \baselineskip14\p@\relax
    {\Large\textit \@title}
  \end{center}%
}
\def\@setauthors{%
  \begingroup
  \def\thanks{\protect\thanks@warning}%
  \trivlist
  \centering\footnotesize \@topsep30\p@\relax
  \advance\@topsep by -\baselineskip
  \item\relax
  \author@andify\authors
  \def\\{\protect\linebreak}%
  {\scshape \authors}%
  \ifx\@empty\contribs
  \else
    ,\penalty-3 \space \@setcontribs
    \@closetoccontribs
  \fi
  \endtrivlist
  \endgroup
}
\makeatother
\newtheorem{thm}{Theorem}[section]
\newtheorem{thm*}{Theorem}

\newtheorem{defn}[thm]{Definition}

\newtheorem*{q*}{Question}

\theoremstyle{definition}
\newtheorem{ex}[thm]{Example}

\newtheorem*{rem*}{Remark}

\newtheorem*{rems*}{Remarks}
\newtheorem{cor*}{Corollary}

\numberwithin{figure}{section}
\numberwithin{table}{section}

\newcommand{\OP}{\operatorname}

\definecolor{dred}{RGB}{200,00,10}

\begin{document}

\title[Concordance of certain 3-braids and Gauss diagrams]{Concordance of certain 3-braids and Gauss diagrams}
\author[Brandenbursky]{Michael Brandenbursky}
\address{Department of Mathematics, Ben Gurion University, Israel}
\email{brandens@math.bgu.ac.il}
\keywords{braids, knots, concordance, Gauss diagrams}
\subjclass[2000]{57}

\begin{abstract}
Let $\beta:=\sigma_1\sigma_2^{-1}$ be a braid in $\B B_3$, where $\B B_3$ is the braid group on 3 strings
and $\sigma_1, \sigma_2$ are the standard Artin generators. 
We use Gauss diagram formulas to show that for each natural number $n$ not divisible by $3$ the knot 
which is represented by the closure of the braid $\beta^n$ is algebraically slice if
and only if $n$ is odd. As a consequence, we deduce some properties of Lucas numbers.
\end{abstract}

\maketitle
\section{Introduction} \label{S:intro}

Let $\OP{Conc(\B S^3)}$ denote the abelian group of concordance classes
of knots in $\B S^3$. Two knots $K_0,K_1\in \B S^3=\partial \B B^4$
are {\em concordant} if there exists a smooth embedding
$c\colon\B S^1\times [0,1]\to \B B^4$ such that
$c(\B S^1\times \{0\})=K_0$ and $c(\B S^1\times \{1\})=K_1$.
The knot is called {\em slice} if it is concordant to the unknot.
The addition in $\OP{Conc}(\B S^3)$ is defined by the connected
sum of knots. The inverse of an element $[K]\in \OP{Conc}(\B S^3)$ is
represented by the knot $-K^*$, where $-K^*$ denotes the mirror image of the
knot $K$ with the reversed orientation. 

Let $\OP{AConc(\B S^3)}$ denote the algebraic concordance group of knots in $\B S^3$. 
The elements of this group are equivalence classes of Seifert forms $[V_F]$ associated
with an arbitrary chosen Seifert surface $F$ of a given knot $K$. 
The addition in $\OP{AConc(\B S^3)}$ is induced by direct sum.
A knot $K$ is called {\em algebraically slice} if it has a Seifert matrix which is metabolic. 
It is a well known fact that every slice knot is algebraically slice. For more information about
these groups see \cite{MR2179265}.

Let $\B B_3$ denote the Artin braid group on $3$ strings and let $\sigma_1, \sigma_2$
be the standard Artin generators of $\B B_3$, i.e. $\sigma_i$ is represented by half-twist of $i+1$-th string 
over $i$-th string and $\B B_3$ has the following presentation
$$\B B_3=\left\langle \sigma_1,\sigma_2|\thinspace \sigma_1\sigma_2\sigma_1=\sigma_2\sigma_1\sigma_2\right\rangle .$$

In this paper we discuss properties of a family of knots in which
every knot is represented by a closure of the braid $\beta^n$, where $\beta=\sigma_1\sigma_2^{-1}\in\B B_3$
and $n\neq 0\OP{mod}3$. This family of braids is interesting in the following sense: the braid $\beta$ is 
the simplest (of the smallest length) non-trivial braid in $\B B_3$ whose stable commutator length is zero. 
Hence by a theorem of Kedra and the author the four ball genus of every knot in this family is bounded by $4$,
see \cite[Section 4.E.]{MR3426696}. 

\begin{thm*}\label{T:main}
Let $n$ be any natural number not divisible by $3$.  Then the closure of $\beta^n$ is of order 2 in $\OP{AConc}(\B S^3)$ if 
$n$ is even and the closure of $\beta^n$ is algebraically slice if $n$ is odd. 
\end{thm*}

We would like to add the following remarks:
\begin{itemize}
\item The above theorem is not entirely new. The fact that the closure of $\beta^n$ is a non-slice knot 
when $n$ is even was proved by Lisca \cite{Lisca} using a celebrated theorem of Donaldson 
(also \cite[Section 6.2]{S} implies the same result). However, our proof of this fact is different.
It uses Gauss diagram technique and is elementary.
\item It is still unknown whether the induced family of smooth or even algebraic concordance classes is infinite, 
and these seem to be hard questions.
\end{itemize}

Let $\{L_n\}_{n=1}^\infty$ be a sequence of Lucas numbers, i.e. it is a Fibonacci sequence with $L_1=1$ and $L_2=3$. 
Surprisingly, Theorem \ref{T:main} has a corollary which is the following number theoretic statement. 

\begin{cor*}\label{C:main}
Let $n\in \B N$. Then
\begin{enumerate}
\item $L_{12n\pm 4}$ is equivalent to $5\OP{mod} 8$ or $7\OP{mod} 8$ 
\item $L_{12n\pm 2}\equiv 3\OP{mod} 8$
\item $L_{12n\pm 2}-2$ is a square.
\end{enumerate}
\end{cor*}

\begin{rem*}
Corollary \ref{C:main} is not new. All parts of it can be proved directly.  
However, we think that it is interesting that a number theoretic result
can be deduced from a purely topological statement. 
We would like to point out that the proof (identical to ours) 
of the fact that $L_{12n\pm2}-2$ is a square for every $n$ 
was given first in \cite[Section 6.2]{S}.
\end{rem*}

\section{Proofs }\label{S:proofs}
Let us recall the notion of a Gauss diagram.

\begin{defn}\rm
Given a classical link diagram $D$, consider a collection of oriented circles parameterizing it. 
Unite two preimages of every crossing of $D$ in a pair and connect them by an arrow, 
pointing from the overpassing preimage to the underpassing one. 
To each arrow we assign a sign (writhe) of the corresponding crossing. 
The result is called the \textit{Gauss diagram} $G$ corresponding to $D$.
\end{defn}

We consider Gauss diagrams up to an orientation-preserving diffeomorphisms of the circles.
In figures we will always draw circles of the Gauss diagram with a counter-clockwise orientation.
A classical link can be uniquely reconstructed from the corresponding Gauss diagram \cite{MR1763963}.
We are going to work with based Gauss diagrams, i.e. Gauss diagrams with a base point
(different from the endpoints of the arrows) on one of the circles.

\begin{ex}\rm
Diagram of the trefoil knot together with the corresponding Gauss diagram is shown in Figure \ref{fig:trefoil+Gauss}.
\end{ex}

\begin{figure}[htb]
\centerline{\includegraphics[height=0.8in]{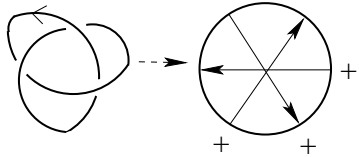}}
\caption{\label{fig:trefoil+Gauss} Diagrams of the trefoil.}
\end{figure}

\begin{proof}[Proof of Theorem \ref{T:main}]
One direction follows immediately from the work of Long \cite{MR735371}, i.e.
if $n$ is odd and not divisible by $3$ then the closure
of the braid $\beta^n$ is strongly plus-amphicheiral
and hence it is algebraically slice.

Let us prove the other direction. For a braid $\alpha$ denote by $\widehat{\alpha}$ its closure.
Let $i=1,2$. By reversing the orientation of all strings in the braid $\beta^{3n+i}$, 
one immediately sees that the knot $-\widehat{\beta^{3n+i}}$ is equivalent to the knot $\left(\widehat{\beta^{3n+i}}\right)^*$. 
Hence the knot $\widehat{\beta^{3n+i}}$ is of order at most 2 in $\OP{Conc}(\B S^3)$ and hence in $\OP{AConc}(\B S^3)$.
To complete the proof we must show that for each odd $n$ the knot $\widehat{\beta^{3n+1}}$ is not algebraically slice and for each even $n$ 
the knot $\widehat{\beta^{3n+2}}$ is not algebraically slice.

Given knot $K$ let $\OP{Arf}(K)$ be the Arf invariant of $K$. Recall that $\OP{Arf}(K):=c_2(K)\OP{mod}2$, where 
$c_2(K)$ is the coefficient before $z^2$ in the Conway polynomial of $K$. It is known that if $c_2(K)\OP{mod}2=1$, 
then the knot $K$ is not algebraically slice, see e.g. \cite{MR811549, MR2179265}.

\textbf{Case 1.} We consider the knot  $\widehat{\beta^{3n+1}}$ where $n$ is odd.
Suppose that $n=1$, then $\widehat{\beta^{3n+1}}=\widehat{\beta^4}$ and $\OP{Arf}(\widehat{\beta^4})=1$. 
The computation of this fact is simple and is left to the reader. It follows that in order to show that for every 
odd $n$ one has $\OP{Arf}(\widehat{\beta^{3n+1}})=1$, it is enough to show that the equality
$$\OP{Arf}(\widehat{\beta^{3n+1}})=\left(\OP{Arf}(\widehat{\beta^{3(n-1)+1}})+1\right)\OP{mod}2$$
holds for every $n$.

\begin{figure}[htb]
\centerline{\includegraphics[height=0.8in]{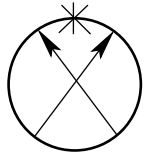}}
\caption{\label{fig:A2} Arrow diagram of Polyak and Viro.}
\end{figure}

It follows from the work of Polyak and Viro \cite{MR1316972} that one can compute $\OP{Arf}(K)$ 
by counting $\OP{mod}2$ an arrow diagram, shown in Figure \ref{fig:A2}, in any Gauss diagram of $K$. 
Let $n$ be any natural number. In Figure \ref{fig:beta-Gauss} we show a diagram 
of a knot $\widehat{\beta^{3n+1}}$ together with a corresponding Gauss diagram $G_n$.

\begin{figure}[htb]
\centerline{\includegraphics[height=4in]{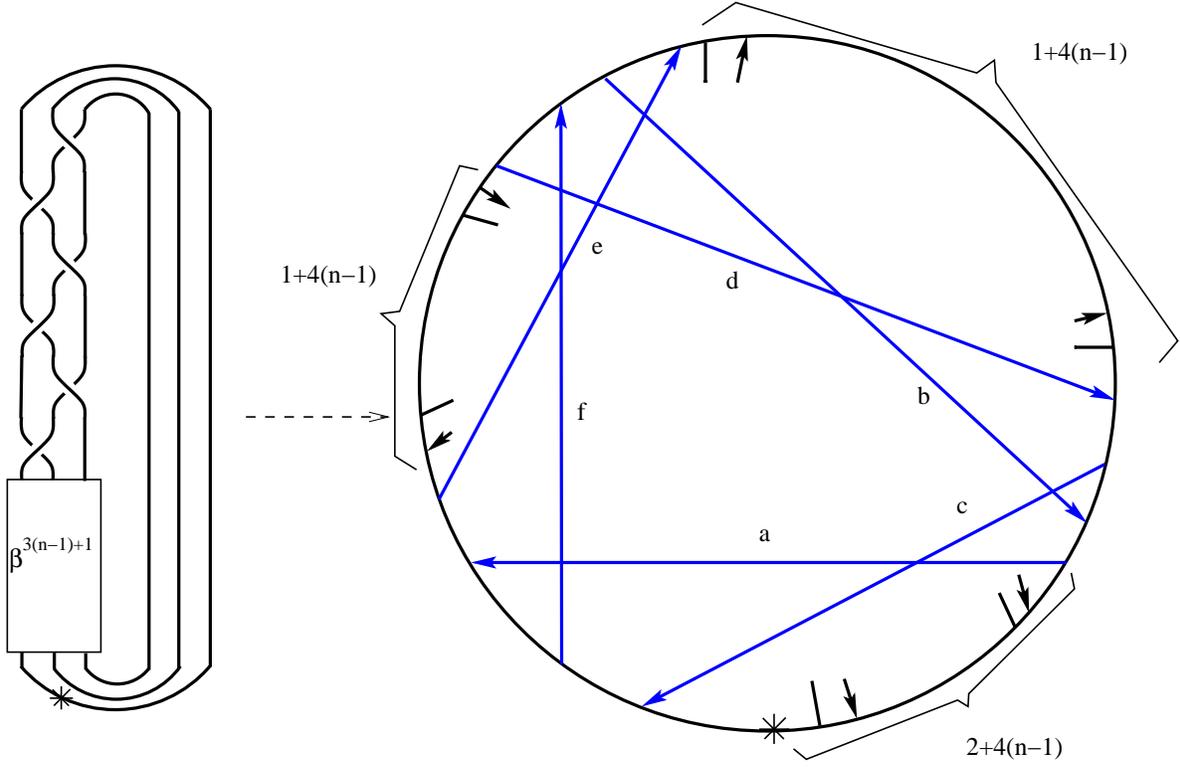}}
\caption{\label{fig:beta-Gauss} Knot and Gauss diagrams of $\widehat{\beta^{3n+1}}$.}
\end{figure}

Denote by $A_2(\widehat{\beta^{3n+1}})$ the number $\OP{mod}2$ of arrow diagrams, shown in Figure \ref{fig:A2}, in $G_n$. Hence $\OP{Arf}(\widehat{\beta^{3n+1}})=A_2(\widehat{\beta^{3n+1}})$. For simplicity we call the arrow diagram presented in Figure \ref{fig:A2} the diagram of type $A$.

1. There is only one type $A$ arrow diagram in $G_n$ which involve only blue arrows. It is a diagram whose arrows are labeled by $b$ and $c$.

2. The number of type $A$ arrow diagrams in $G_n$ which involve one black arrow and a blue arrow labeled by $a$ equals to the number of type $A$ arrow diagrams in $G_n$ which involve one black arrow and a blue arrow labeled by $c$.

3. The number of type $A$ arrow diagrams in $G_n$ which involve one black arrow and a blue arrow labeled by $b$ equals to the number of type $A$ arrow diagrams in $G_n$ which involve one black arrow and a blue arrow labeled by $d$.

4. There are no type $A$ arrow diagrams in $G_n$ which involve one black arrow and a blue arrow labeled by $e$ or by $f$.

5. By removing blue arrows from a Gauss diagram of $\widehat{\beta^{3n+1}}$, we get a Gauss diagram of $\widehat{\beta^{3(n-1)+1}}$.

Claims 1--5 yield the equality

$$A_2(\widehat{\beta^{3n+1}})=\left(A_2(\widehat{\beta^{3(n-1)+1}})+1\right)\OP{mod}2,$$
which concludes the proof of case 1.

\textbf{Case 2.} We consider the knot  $\widehat{\beta^{3n+2}}$ where $n=2k$ for $k\geq 0$.
Suppose that $n=0$, then $\widehat{\beta^{3n+2}}=\widehat{\beta^2}$ which is the figure eight knot 
and so $\OP{Arf}(\widehat{\beta^2})=1$. It follows that in order to show that for every 
even $n$ one has $\OP{Arf}(\widehat{\beta^{3n+2}})=1$, it is enough to show that the equality
$$\OP{Arf}(\widehat{\beta^{3n+2}})=\left(\OP{Arf}(\widehat{\beta^{3(n-1)+2}})+1\right)\OP{mod}2$$
holds for every $n$.

In Figure \ref{fig:beta-Gauss-even} we show a diagram 
of a knot $\widehat{\beta^{3n+1}}$ together with a corresponding Gauss diagram $G'_n$.
As explained above  $\OP{Arf}(\widehat{\beta^{3n+2}})=A_2(\widehat{\beta^{3n+2}})$. 

\begin{figure}[htb]
\centerline{\includegraphics[height=4in]{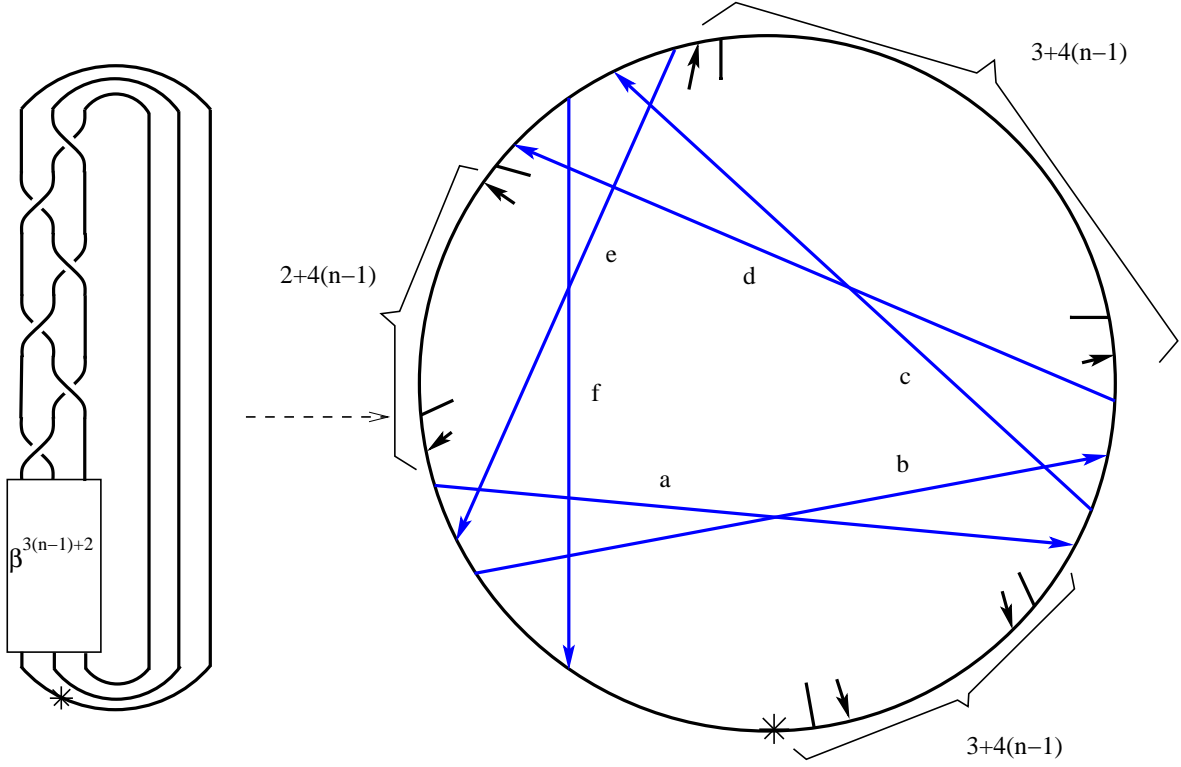}}
\caption{\label{fig:beta-Gauss-even} Knot and Gauss diagrams of $\widehat{\beta^{3n+2}}$.}
\end{figure}

1. There are three type $A$ arrow diagrams in $G'_n$ which involve only blue arrows. These are diagrams whose arrows are labeled by $a$ and $f$, $b$ and $f$, and $a$ and $e$.

2. The number of type $A$ arrow diagrams in $G'_n$ which involve one black arrow and a blue arrow labeled by $e$ equals to the number of type $A$ arrow diagrams in $G'_n$ which involve one black arrow and a blue arrow labeled by $f$.

3. The number of type $A$ arrow diagrams in $G'_n$ which involve one black arrow and a blue arrow labeled by $d$ equals to the number of type $A$ arrow diagrams in $G'_n$ which involve one black arrow and a blue arrow labeled by $c$.

4. There are no type $A$ arrow diagrams in $G'_n$ which involve one black arrow and a blue arrow labeled by $a$ or by $b$.

5. By removing blue arrows from a Gauss diagram of $\widehat{\beta^{3n+2}}$, we get a Gauss diagram of $\widehat{\beta^{3(n-1)+2}}$.

Claims 1--5 yield the equality

$$A_2(\widehat{\beta^{3n+2}})=\left(A_2(\widehat{\beta^{3(n-1)+2}})+1\right)\OP{mod}2,$$
which concludes the proof of case 2 and the proof of the theorem.
\end{proof}

\begin{proof}[Proof of Corollary \ref{C:main}]
It follows from the matrix-tree theorem that the determinant $\OP{det}(K)$of an alternating knot $K$ 
equals to the number of spanning trees in the associated Tait graph.
Note that the knots $\widehat{\beta^{3n+1}}$, $\widehat{\beta^{3n+2}}$ are alternating for each $n$ 
and hence their Tait graphs are Wheel graphs. It follows from \cite{MR0354532} that 
the number of spanning trees in the Wheel graph on $n+1$ points equals to $L_{2n}-2$ 
\footnote{This fact was communicated to the author by Brendan Owens.}. Hence
\begin{equation}\label{E:Lucas}
\OP{det}(\widehat{\beta^{n}})=L_{2n}-2
\end{equation}
for every $n$ not divisible by $3$.
It follows from Theorem \ref{T:main} that the knots $\widehat{\beta^{6n+1}}$, $\widehat{\beta^{6n-1}}$ are algebraically slice for each $n$.
Since the determinant of an algebraically slice knot is a square we conclude that
\begin{itemize}
\item $L_{12n\pm2}-2$ is a square for every $n$.
\end{itemize} 

In \cite{MR0238301} Murasugi proved that $\OP{Arf}(K)=0\Leftrightarrow \OP{det}(K)\equiv \pm 1\OP{mod}8$.
It follows that $L_{12n\pm2}$ is congruent to $1\OP{mod}8$ or $3\OP{mod}8$. But since a square number
can not be congruent to $-1\OP{mod}8$ we obtain
\begin{itemize}
\item $L_{12n\pm2}-2\equiv3\OP{mod}8$.
\end{itemize} 
In the proof of Theorem \ref{T:main} we showed that the Arf invariant of knots $\widehat{\beta^{6n+2}}$, $\widehat{\beta^{6n-2}}$ equals to $1$ for each $n$. 
It follows from Murasugi result that $\OP{det}(\widehat{\beta^{6n\pm 2}})\equiv 3\OP{mod}8$ or $\OP{det}(\widehat{\beta^{6n\pm 2}})\equiv 5\OP{mod}8$. By \eqref{E:Lucas}
we get 
\begin{itemize}
\item $L_{12n\pm 4}$ is equivalent to $5\OP{mod} 8$ or $7\OP{mod} 8$ 
\end{itemize}
which concludes the proof of the corollary.
\end{proof}  

\subsection*{Acknowledgments}
The author would like to thank Steve Boyer, Jarek Kedra, Ana Garcia Lecuona, 
Paolo Lisca and Brendan Owens for useful conversations.

This work was initiated during author's stay in CRM-ISM Montreal. 
The author was partially supported by the CRM-ISM fellowship. He would like
to thank CRM-ISM Montreal for the support and great research atmosphere. 

\bibliography{bibliography}
\bibliographystyle{acm}

\end{document}